\documentclass{amsart}
\usepackage{amssymb}
\usepackage{amsmath}
\usepackage{tikz-cd}
\usepackage{amsthm}
\theoremstyle{plain}

\theoremstyle{plain}
\newtheorem{thm}{Theorem}
\newtheorem{prop}{Proposition}
\newtheorem{lem}{Lemma}
\newtheorem{coro}{Corollary}

\usepackage{geometry}
\theoremstyle{definition}
\newtheorem{exam}{Example}
\newtheorem{rem}{Remark}

\begin{document}
\setcounter{page}{1}

\title{ Spaces of $u\tau$-Dunford-Pettis and $u\tau$-Compact Operators on Locally Solid Vector Lattices}
\author[Nazife Erkur\c{s}un-\"Ozcan, Niyazi An{\i}l Gezer, Omid Zabeti ]{Nazife Erkur\c{s}un-\"Ozcan$^{(1)}$, Niyazi An{\i}l Gezer$^{(2)}$, Omid Zabeti$^{(3, *)}$}

\address{$^{1}$ Department of Mathematics, Faculty of Science, Hacettepe University, Ankara, 06800,Turkey.}
\email{{erkursun.ozcan@hacettepe.edu.tr}}

\address{$^{2}$ Department of Mathematics, Middle East Technical University, Ankara, 06800, Turkey.}
\email{{ngezer@metu.edu.tr }}

\address{$^{3}$ Department of Mathematics, Faculty of Mathematics, University of Sistan and Baluchestan, P.O. Box: 98135-674, Zahedan, Iran.}
\email{{o.zabeti@gmail.com}}

\subjclass[2010]{Primary: 46A40, 54A20; Secondary:  46B40, 46A03.}

\keywords{$u\tau$-convergence, $u\tau$-Dunford-Pettis operator, $u\tau$-compact operator, locally solid vector lattice.}

\date{Received: xxxxxx; Revised: yyyyyy; Accepted: zzzzzz.
\newline \indent $^{*}$ Corresponding author}

\begin{abstract}
Suppose $X$ is a locally solid vector lattice. It is known that there are several non-equivalent spaces of bounded operators on $X$. In this paper, we consider some situations under which these classes of bounded operators form  locally solid vector lattices. In addition, we generalize some notions of $uaw$-Dunford-Pettis operators and $uaw$-compact operators defined on a Banach lattice to general theme of locally solid vector lattices. With the aid of appropriate topologies, we investigate some relations between topological and lattice structures of these operators. In particular, we characterize those spaces for which these concepts of operators and the corresponding classes of bounded ones coincide.

\end{abstract}
\maketitle
\section{Introduction and preliminaries}
Recently, many papers are devoted to the concept of unbounded convergence (see \cite{DEM2, DOT, G, GTX, GX, KMT, T, Tr, Z}). The notion of $uo$-convergence was proposed firstly in \cite{N} and considered more in \cite{DM}. The concept of $un$-convergence was introduced in \cite{Tr} and further investigated in \cite{DOT, KMT}. Unbounded convergent nets in terms of weak convergence, called $uaw$-convergence, was introduced by Zabeti and investigated in \cite{Z}. All these notions are defined on Banach lattices. For other necessary terminology on vector and Banach lattice, we refer to \cite{AB1, AB}.

Throughout the present paper, unbounded topology is considered on a locally solid vector lattice. The pair $(X,\tau)$ stands for a locally solid vector lattice, whereas the pair $(Y,\tau')$ denotes a generic locally convex space. Following \cite{DEM2}, we write $x_{\alpha}\xrightarrow{u\tau}x$ for a net $(x_\alpha)$ in a locally solid vector lattice $(X,\tau)$ if $|x_\alpha -x|\wedge w\xrightarrow{\tau} 0$ for all $w\in X_{+}$; this notion was also discovered independently in \cite{T}. We say that the net $(x_{\alpha})$ is unbounded $\tau$-convergent to $x$ whenever $x_{\alpha}\xrightarrow{u\tau} x$. For more expositions on this concept and the related topics, see \cite{DEM2, T}.

In \cite{Tr1} , a spectral theory for bounded operators between topological vector spaces was developed. Various results on different classes of bounded operators on topological vector spaces were obtained. Among those bounded operators, the spaces of $nb$-bounded and $bb$-bounded operators were considered and many properties were investigated.

In prior to anything, let us recall some notions of bounded operators between topological vector spaces. Let $X$ and $Y$ be  topological vector spaces. A linear operator $T$ from $X$ into $Y$ is said to be $nb$-bounded if there is a zero neighborhood $U\subseteq X$ such that $T(U)$ is bounded in $Y$. $T$ is called $bb$-bounded if for each bounded set $B\subseteq X$, $T(B)$ is bounded. These concepts are not equivalent; more precisely, continuous operators are, in a sense,  in the middle of these notions of bounded operators, but in a normed space, these concepts have the same meaning( see \cite{Tr1,Km1,Z0} for more details on this topic).
The class of all $nb$-bounded operators on a topological vector space $X$ is denoted by $B_{n}(X)$ and is equipped with the topology of uniform convergence on some zero neighborhood, namely, a net $(S_{\alpha})$ of $nb$-bounded operators converges to zero  on some zero neighborhood $U\subseteq X$ if for any zero neighborhood $V\subseteq X$ there is an $\alpha_0$ such that $S_{\alpha}(U) \subseteq V$ for each $\alpha\geq\alpha_0$. The class of all $bb$-bounded operators on $X$ is denoted by $B_{b}(X)$ and is allocated to the topology of uniform convergence on bounded sets. Recall that a net $(S_{\alpha})$ of  $bb$-bounded operators uniformly converges to zero  on a bounded set $B\subseteq X$ if for any zero neighborhood $V \subseteq X$ there is an $\alpha_0$ with $S_{\alpha}(B) \subseteq V$ for each $\alpha\geq\alpha_0$.

The class of all continuous operators on $X$  is denoted by $B_c(X)$ and is equipped with the topology of equicontinuous convergence, namely, a net $(S_{\alpha})$ of continuous operators converges equicontinuously to zero if for each zero neighborhood $V$ there is a zero neighborhood $U$ such that for every $\varepsilon>0$ there exists an $\alpha_0$ with $S_{\alpha}(U)\subseteq \varepsilon V$ for each $\alpha\geq\alpha_0$. See \cite{Tr1} for more information on these classes of operators. In general, we have $B_n(X)\subseteq B_c(X)\subseteq B_b(X)$ and when $X$ is locally bounded, they coincide.

In the present paper, beside lattice structures for bounded operators, the concepts of $u\tau$-Dunford-Pettis and $u\tau$-compact operators are studied and some properties of them are investigated. Specially, we consider some situations under which, these concepts with the corresponding spaces of bounded operators agree. Now, we have the following.

Suppose $(X,\tau)$ is a locally solid vector lattice and $(Y,\tau')$ is a locally convex space. An operator $T\colon (X,\tau)\rightarrow (Y,\tau')$ is said to be {\bf $u\tau$-Dunford-Pettis} if for every $\tau$-bounded net $(x_{\alpha})$ in $X$, $x_{\alpha}\xrightarrow{u\tau}0$ implies $T(x_{\alpha})\xrightarrow{\tau'} 0$ in $Y$. Denote by $DP_{u\tau}(X,Y)$ the linear space generated by $u\tau$-Dunford-Pettis operators.

It is known that in topological vector space setting, we have different non-equivalent concepts for bounded operators and compact ones between them. Consequently, one can have two variant of unbounded compact operators on locally solid vector lattices.

An $nb$-bounded operator $T\colon (Y,\tau')\to (X,\tau)$ is called {\bf $nu\tau$-compact} if there exists a zero neighborhood $U\subseteq Y$ such that the set $T(U)$ is $u\tau$-relatively compact in $X.$ Denote by $K_{nu\tau}(Y,X)$ the linear space generated by $nu\tau$-compact operators from $Y$ into $X$.

A $bb$-bounded operator $T\colon (Y,\tau')\to (X,\tau)$ is said to be $bu\tau$-compact if for every bounded set $B\subseteq Y$,  the set $T(B)$ is $u\tau$-relatively compact in $X$. We denote by $K_{bu\tau}(Y,X)$, the class of all $bu\tau$-compact operators from $Y$ into $X$.

In many cases, it is useful to consider sequential versions of these operators. In details, an operator $T\colon (X,\tau)\rightarrow (Y,\tau')$ is said to be sequentially $u\tau$-Dunford-Pettis if for every $\tau$-bounded sequence $(x_n)$ in $X$, $x_n\xrightarrow{u\tau}0$ implies $T(x_n)\xrightarrow{\tau'} 0$ in $Y$. A variant of the statement is investigated in \cite{EGZ}. An $nb$-bounded operator $T\colon (Y,\tau')\to (X,\tau)$ is said to be sequentially $nu\tau$-compact if there exists a zero neighborhood $U\subseteq Y$ such that for every sequence $(x_n)$ in $U$ the sequence $(T(x_n))$ has a $u\tau$-convergent subsequence in $X$. Similarly, $bb$-bounded operator $T$ is called sequentially $bu\tau$-compact if for every sequence $(x_n)$ in $B$, in which $B$ is a bounded set in $Y$, the sequence $(T(x_n))$ has a $u\tau$-convergent subsequence in $X$; see \cite{EGZ, KMT} for more information on these notions. It can be easily seen that many results of this note can be considered for sequential case, directly. As usual, we denote by $DP_{u\tau}(X)$, $K_{nu\tau}(X)$, and $K_{bu\tau}(X)$, the space of all $u\tau$-Dunford-Pettis, $nu\tau$-compact, and $bu\tau$-compact operators on a locally solid vector lattice $X$, respectively.

This paper is devoted to investigate  $u\tau$-Dunford-Pettis and $u\tau$-compact operators on locally solid vector lattices. Unbounded absolute weak Dunford-Pettis and unbounded absolute weak compact operators are studied in \cite{EGZ}. Moreover, $p$-compactness and $up$-compactness on lattice normed spaces are investigated in \cite{AEEM2}. Some results presented in this paper are hold in the $um$-case on locally convex vector lattices; for details see \cite{DEM1}.

\begin{rem}
It is known that every order bounded operator from a Banach lattice to a normed lattice is continuous. But in general locally solid vector lattices, this may fail; suppose $X$ is $\ell_{\infty}$ with the weak topology and $Y$ is $\ell_{\infty}$ with the usual norm topology.  Consider the identity operator from $X$ into $Y$. Indeed, it is order bounded but not continuous. Also, not every order bounded operator is automatically bounded. Let $X$ be ${\Bbb R}^{\Bbb N}$, the space of all real sequences with the product topology and the coordinate-wise order. Suppose $I$ is the identity operator on $X$; clearly, it is order bounded but not $nb$-bounded. Finally, suppose $X$ is $c_{00}$ with pointwise ordering and the usual norm topology. Then the operator $T$ on $X$ which maps every $(x_n)$ into $(nx_n)$ is order bounded but certainly not $bb$-bounded.
\end{rem}

It is natural we can not expect order properties from bounded operators between topological vector spaces but there are good news if we restrict our attention to order bounded topologically bounded operators between locally solid vector lattices.

\begin{lem}\label{301}
Suppose $X$ is a Dedekind complete locally solid vector lattice with  Fatou topology. Then the following statements hold.
\begin{itemize}
\item[\em i.] {$B^{b}_{n}(X)$, the space of all order bounded $nb$-bounded operators on $X$, is a vector lattice}.
\item[\em ii.]{$B^{b}_{c}(X)$, the space of all order bounded continuous operators on $X$, is a vector lattice}.
\item[\em iii.] {$B^{b}_{b}(X)$, the space of all order bounded $bb$-bounded operators on $X$, is a vector lattice}.
\end{itemize}
\end{lem}
\begin{proof}
$(i)$. It suffices to prove that for an operator $T\in B^{b}_{n}(X)$, $T^+ \in B^{b}_{n}(X)$. By the Riesz-Kantorovich formula, we have
\[T^+(x)=\sup \{T(u), 0\leq u\leq x\}.\]
Choose zero neighborhood $U\subseteq X$ such that $T(U)$ is bounded. So, for arbitrary zero neighborhood $V$, there is a positive $\gamma$ with $T(U)\subseteq \gamma V$. Fix $x\in U_{+}$. Using solidness of $U$, for each positive $u\leq x$, we have $u\in U_{+}$. Furthermore, since $V$ is order closed because of the Fatou's property, $T(u)\in \gamma V$ implies that $T^{+}(x)\in \gamma V$. Thus, we see that $T^+(U)$ is also bounded.

$(ii)$. Similar to part $(i)$.

$(iii)$. It is similar to part $(i)$; just note that for a bounded set $B\subseteq X$, we can assume that $B$ is solid, otherwise, one may consider the solid hull of $B$ which is indeed bounded.
\end{proof}
The following results extend \cite[Theorem 6 and Theorem 7]{Z1} to a more general setting.
\begin{thm}\label{600}
Suppose $X$ is a Dedekind complete locally solid vector lattice with  Fatou topology. Then $B^{b}_{n}(X)$ is locally solid with respect to the uniform convergence topology on some zero neighborhood.
\end{thm}
\begin{proof}
Let $T\in B^{b}_{n}(X)$ and $x\in X_{+}$. By the Riesz-Kantorovich theorem, we have
 \[T^{+}(x)=\sup\{T(u), 0\leq  u\leq x\}.\]
   Now, suppose $(T_{\alpha})$ is a  net of order bounded $nb$-bounded operators that converges uniformly on some zero neighborhood $U\subseteq X$ to the linear operator $T$ in $ B^{b}_{n}(X)$. Choose arbitrary zero neighborhood $V\subseteq X$. Fix $x\in U_{+}$. Note that for every element $0\leq u\leq x$ we have $u\in U_{+}$ since $U$ is solid. Recall that for two subsets $A,B$ in a vector lattice, we have $\sup(A)-\sup(B)\le\sup(A-B)$. Thus,
 \[\sup\{T_{\alpha}(u): u\in X_{+}, u\leq x\}-\sup\{T(u): u\in X_{+}, u\leq x\}\]
 \[\le\sup\{(T_{\alpha}-T)(u):u\in X_{+}, u\leq x\}.\]
 There exists an $\alpha_0$ such that $(T_{\alpha}-T)(U)\subseteq V$ for each $\alpha\geq\alpha_0$.
Therefore, from order closedness of zero neighborhood $V$, we obtain
\[{T_{\alpha}}^{+}(x)-T^{+}(x)\leq(T_{\alpha}-T)^{+}(x)\in V.\]

Now, using \cite[Theorem 2.17]{AB1}, yields the desired result.
\end{proof}
\begin{thm}
Suppose $X$ is a Dedekind complete locally solid vector lattice with Fatou topology. Then $B^{b}_{c}(X)$ is locally solid with respect to the equicontinuous convergence topology.
\end{thm}
\begin{proof}
The proof is similar to what we had in Theorem \ref{600}.
Let $T\in B^{b}_{c}(X)$ and $x\in X_{+}$. By the Riesz-Kantorovich theorem, we have
 \[T^{+}(x)=\sup\{T(u), 0\leq  u\leq x\}.\]
   Now, suppose $(T_{\alpha})$ is a  net of order bounded continuous operators which is convergent equicontinuously to the linear operator $T$ in $ B^{b}_{c}(X)$. Choose arbitrary zero neighborhood $V\subseteq X$. There exists zero neighborhood $U\subseteq X$ such that for each $\varepsilon>0$ we have $(T_{\alpha}-T)(U)\subseteq \varepsilon V$ for sufficiently large $\alpha$. Fix $x\in U_{+}$. Since $U$ is solid, every positive element dominated by $x$, lies in $U_{+}$. Recall that for two subsets $A,B$ in a vector lattice, we have $\sup(A)-\sup(B)\le\sup(A-B)$. Thus,
 \[\sup\{T_{\alpha}(u): u\in X_{+}, u\leq x\}-\sup\{T(u): u\in X_{+}, u\leq x\}\]
 \[\le\sup\{(T_{\alpha}-T)(u):u\in X_{+}, u\leq x\}.\]
Therefore, the following inequality happens since zero neighborhood $V$ has the Fatou's property by the assumption:
\[{T_{\alpha}}^{+}(x)-T^{+}(x)\leq(T_{\alpha}-T)^{+}(x)\in \varepsilon V.\]
\end{proof}
\begin{thm}
Suppose $X$ is a Dedekind complete locally solid vector lattice with Fatou topology. Then $B^{b}_{b}(X)$ is locally solid with respect to the uniform convergence topology on bounded sets.
\end{thm}
\begin{proof}
Let $T\in B^{b}_{b}(X)$ and $x\in X_{+}$. By the Riesz-Kantorovich theorem, we have
 \[T^{+}(x)=\sup\{T(u), 0\leq  u\leq x\}.\]
   Now, suppose $(T_{\alpha})$ is a  net of order bounded $bb$-bounded operators that converges uniformly on bounded sets to the linear operator $T$ in $ B^{b}_{b}(X)$. Fix bounded set $B\subseteq X$. In view of Lemma \ref{301}$(iii)$, we can assume that $B$ is solid. Choose arbitrary zero neighborhood $V\subseteq X$. Fix $x\in B_{+}$. Recall that for two subsets $A,C$ in a vector lattice, we have $\sup(A)-\sup(C)\le\sup(A-C)$. Thus,
 \[\sup\{T_{\alpha}(u): u\in X_{+}, u\leq x\}-\sup\{T(u): u\in X_{+}, u\leq x\}\]
 \[\le\sup\{(T_{\alpha}-T)(u):u\in X_{+}, u\leq x\}.\]
 There exists an $\alpha_0$ such that $(T_{\alpha}-T)(B)\subseteq V$ for each $\alpha\geq\alpha_0$.
Consider this point that $B$ is solid and $V$ is order closed by the hypothesis, so that
\[{T_{\alpha}}^{+}(x)-T^{+}(x)\leq(T_{\alpha}-T)^{+}(x)\in V.\]
\end{proof}
\begin{rem}
One may verify that $K_n(X)\subseteq K_{nu\tau}(X)\subseteq B_n(X)$ and $K_b(X)\subseteq K_{bu\tau}(X)\subseteq B_b(X)$, where $K_n(X)$ and $K_b(X)$ stand for the class of all $n$-compact ( $b$-compact) operators on locally solid vector lattice $X$ ( see \cite{Km1} for some observation on this topic). If $X$ has the Heine-Borel property by \cite[Proposition 2.5 and Remark 2.6]{Km1} these concepts agree.
\end{rem}
Now, we consider some ideal properties for these spaces of operators.
\begin{prop} \label{101}
Let $S\colon (X,\tau)\rightarrow (Y,\tau')$ and $T\colon (Y,\tau')\rightarrow (Z,\tau'')$  be two operators between locally solid vector lattices.
\begin{itemize}
\item[\em i.] {If $T$ is $nu\tau$-compact  and $S$ is $nb$-bounded then $TS$ is $nu\tau$-compact}.
\item[\em ii.]{If $T$ is $bu\tau$-compact  and $S$ is $bb$-bounded then $TS$ is $bu\tau$-compact}.
\item[\em iii.] {If $T$ is a $u\tau$-Dunford-Pettis operator and $S$ is $bu\tau$-compact then $TS$ is $b$-compact}.
\item[\em iv.]{ If $T$ is continuous and $S$ is $u\tau$-Dunford-Pettis, then $TS$ is $u\tau$-Dunford-Pettis}.
\end{itemize}
\end{prop}
\begin{proof}
$(i)$. Suppose $U\subseteq X$ and $V\subseteq Y$ are zero neighborhoods such that $S(U)$ is bounded in $Y$ and $T(V)$ is $u\tau''$-relatively compact in $Z$. There is some positive $\gamma$ with $S(U)\subseteq \gamma V$, so that $TS(U)\subseteq \gamma T(V)$. This implies that $TS(U)$ is $u\tau''$-relatively compact.

$(ii)$. Fix bounded set $B\subseteq X$. Since $S(B)$ is bounded in $Y$, by assumption, $TS(B)$ is $u\tau''$-relatively compact in $Z$.

$(iii)$. Suppose $(x_{\alpha})$ is a bounded net in $X$. There is a subnet $(y_{\beta})$ such that $S(y_{\beta})\xrightarrow{u\tau'}y$ for some $y\in Y$. Thus, by the hypothesis, $T(S(y_{\beta}))\xrightarrow
{\tau''} T(S(y))$, as desired.

$(iv)$. Suppose $(x_{\alpha})$ is a bounded $u\tau$-null net in $X$. By the assumption, $S(x_{\alpha})\xrightarrow{\tau'}0$. By the assumption, $T(S(x_{\alpha}))$ is topologically null.

\end{proof}

In this part, we investigate some conditions for which these spaces of operators agree.

Suppose $(X,\tau)$ is a locally solid vector lattice in which every convergent net is eventually bounded. This property almost happens in many known spaces; for example metrizable spaces, normed spaces, weak topology, and specially, when we can consider sequences. In this case, one may verify easily that every $u\tau$-Dunford-Pettis operator is automatically continuous but the converse is not true, in general; consider the identity operator on $\ell_1$ when it is equipped with the norm topology. From now on, while we are dealing with $u\tau$-Dunford-Pettis operators and continuous ones, together, we assume that the mentioned topology has this mild property. Recall that a topology $\tau$ on a locally solid vector lattice $X$ is said to be unbounded if $\tau=u\tau$ ( see \cite{T}, Definition 2.7). In this step, we consider a notion named "{\bf boundedly unbounded}" for a locally solid topology $\tau$. We say that a locally solid topology $\tau$ on a vector lattice $X$ is boundedly unbounded if $\tau=u\tau$ in every bounded subset of $X$. Note that boundedly unboundedness and unboundedness differ in general. Consider $X=c_0$ and $\tau=|\sigma|(X,X^*)$. Using \cite[Theorem 7]{Z}, we conclude that $\tau$ is boundedly unbounded but not unbounded since the sequence $(a_n)\subseteq X$ defined via $a_n=(0,\ldots,0,n,0,\ldots)$ where $n$ is in the $n^{th}$-place is $u\tau$-null but not $\tau$-null. In fact, it can be seen using \cite[Theorem 7]{Z} that absolute weak topology on a Banach lattice $X$ is boundedly unbounded if and only if $X^*$ is order continuous. Now, we have the following.
\begin{thm}
$DP_{u\tau}(X)=B_c(X)$ if and only if $X$ is boundedly unbounded.
\end{thm}
\begin{proof}
Suppose $X$ is boundedly unbounded. Therefore, every bounded $u\tau$-null net is $\tau$-null. This means that the identity operator is $u\tau$-Dunford-Pettis. So, by Proposition \ref{101}$(iv)$, we see that $DP_{u\tau}(X)=B_c(X)$. Now, suppose $DP_{u\tau}(X)=B_c(X)$. Therefore, the identity operator lies in $DP_{u\tau}(X)$. So, $X$ is boundedly unbounded.
\end{proof}
When we focus on norm topology, we obtain more familiar result.
\begin{prop}
$DP_{un}(X)=B(X)$ if and only if $X$ has a strong unit.
\end{prop}
\begin{proof}
Suppose $X$ has a strong unit. By \cite[Theorem 2.3]{KMT}, $un$-topology and norm topology agree on $X$, so $DP_{un}(X)=B(X)$. For the converse, assume that $DP_{un}(X)=B(X)$. Thus, the identity operator $I$ lies in $DP_{un}(X)$. Thus, every norm bounded $un$-null net is norm null. Using of \cite[Lemma 2.1 and Lemma 2.2]{KMT}, yields the desired result.
\end{proof}
\begin{thm}
$K_{bu\tau}(X)=B_b(X)$ if and only if $X$ is atomic and has both the Levi and Lebesgue properties.
\end{thm}
\begin{proof}
Suppose $X$ is an atomic locally solid vector lattice with Levi and Lebesgue properties. Then by \cite[Theorem 6]{DEM1}, the identity operator is $bu\tau$-compact. So, by Proposition \ref{101}$(ii)$,  $K_{bu\tau}(X)=B_b(X)$. Now, suppose $K_{bu\tau}(X)=B_b(X)$. Therefore, the identity operator lies in $K_{bu\tau}(X)$. Therefore, every bounded subset of $X$ is $u\tau$-relatively compact. Again, using \cite[Theorem 6]{DEM1}, yields the desired result.
\end{proof}
\begin{coro}\label{500}
Suppose $X$ is a Banach lattice. Then, $K_{un}(X)=B(X)$ and $K_{uaw}(X)=B(X)$ if and only if $X$ is an atomic $KB$-space.
\end{coro}
\begin{rem}
We do not know whether or not $K_{nu\tau}(X)=B_n(X)$, in general. A sufficient condition is that $X$ has the Heine-Borel property. However, this condition is not necessary as $X=\ell_1$ does not have the Heine-Borel property; nevertheless $K_{nu\tau}(X)=B_n(X)$ by Corollary \ref{500}.
\end{rem}

Now, we investigate whether unbounded Dunford-Pettis operators or unbounded compact ones are topologically closed with respect to the induced topologies from corresponding classes of bounded operators. Also, we consider order closedness property for them.

The class of all $nu\tau$-compact ( $bu\tau$-compact) operators are not closed in the corresponding class of bounded operators, respectively, in spite of this fact that sequentially $un$-compact operators are closed; see \cite[Proposition 9.2]{KMT}. Consider the following example. In addition, these spaces are not order closed; for sequentially $un$-compact operators, it is shown in \cite[Example 9.5]{KMT}.
\begin{exam}
Assume that $X$ is $c_0$ with the norm topology and $Y$ is $c_0$ with the weak topology. Let $P_n$ be the projection on the first $n$-components from $X$ into $Y$. Each $P_n$ is compact in both unbounded senses. In addition, $(P_n)$ converges uniformly on the unit ball to the identity operator $I$ from $X$ into $Y$. $I$ is neither $nu\tau$-compact nor $bu\tau$-compact since the sequence $(u_n)$ defined via $n$ one terms at first and null in the sequel is neither norm convergent nor weak convergent in $c_0$ by Dini Theorem \cite[Theorem 3.52]{AB}. Since $c_0$ is boundedly unbounded, it is not also $uaw$-convergent. Finally, note that $P_n\uparrow I$, we conclude that these classes of compact operators are not order closed.
\end{exam}
\begin{prop}
Let $X$ be a locally solid vector lattice. Then $DP_{u\tau}(X)$ is a closed subalgebra of $B_{c}(X)$.
\end{prop}
\begin{proof}
One can see easily that $DP_{u\tau}(X)$ is an algebra. Suppose $(S_{\alpha})$ is a net of $u\tau$-Dunford-Pettis operators which is uniformly convergent equicontinuously to the continuous operator $S$. Choose zero neighborhood $W\subseteq X$, arbitrary. There is a zero neighborhood $V\subseteq X$ with $V+V\subseteq W$. Choose zero neighborhood $U\subseteq X$ such that for each $\varepsilon>0$ there is an index $\alpha_0$ with $(S_{\alpha}-S)(U)\subseteq \varepsilon V$ for each $\alpha\geq\alpha_0$ so that $S(U)\subseteq S_{\alpha}(U)+\varepsilon V$. Assume $(x_{\beta})$ is a bounded $u\tau$-null net in $X$. Find positive scalar $\gamma$ with $(x_{\beta})\subseteq \gamma U$. Corresponding to $\varepsilon=\frac{1}{\gamma}$, we have $(S_{\alpha}-S)(U)\subseteq \frac{1}{\gamma} V$ for sufficiently large $\alpha$ so that $(S_{\alpha}-S)(x_{\beta})\subseteq V$. Fix an $\alpha$. There exists a $\beta_0$ with $S_{\alpha}(x_{\beta})\subseteq V$ for each $\beta\geq\beta_0$. This concludes that $S(x_{\beta})\subseteq W$ for sufficiently large $\beta$.
\end{proof}
Note that we have seen in \cite[Example 8]{EGZ} unbounded absolute weak Dunford-Pettis operators are not order closed, in general.

In view of \cite[Proposition 6]{EGZ}, we have a Kantorovich-like extension for unbounded Dunford-Pettis operators.
\begin{prop} Let $T\colon X\rightarrow Y$ be a positive $u\tau$-Dunford-Pettis operator between locally solid vector lattices with $Y$ Dedekind complete. Then the Kantorovich-like extension $S\colon X\rightarrow Y$ defined via
\begin{center}
$S(y)=\sup \left\{ T(y\wedge y_{\alpha})\colon  (y_{\alpha})\subseteq X_{+},  y_{\alpha}\xrightarrow{u\tau} 0\right\}$
\end{center}
for every $y\in X_{+}$ is again $u\tau$-Dunford-Pettis.
\end{prop}
\begin{proof}
Suppose $y,z\in X_{+}$. Then
\begin{center}
$S(y+z)=\sup_{\beta} \{T((y+z)\wedge \gamma_{\beta})\}\le \sup_{\beta}\{T(y\wedge \gamma_{\beta}\}+\sup_{\beta}\{T(z\wedge \gamma_{\beta})\}\le S(y)+S(z)$,
\end{center}
in which, $(\gamma_{\beta})$ is a positive net that is $u\tau$-null. On the other hand,
\[T(y\wedge a_{\alpha})+T(z\wedge b_{\beta})=T(y\wedge a_{\alpha}+z\wedge b_{\beta})\leq T((y+z)\wedge (a_{\alpha}+b_{\beta}))\leq S(y+z),\]
provided that two positive nets $(a_{\alpha}),(b_\beta)$ are $u\tau$-null so that $S(y)+S(z)\leq S(y+z)$. Therefore, by the Kantorovich extension Theorem \cite[Theorem 1.10]{AB}, $S$ extends to a positive operator. Denote by $S$ the extended operator $S\colon X\rightarrow Y.$

We show that $S$ is also $u\tau$-Dunford-Pettis. Suppose bounded net $(y_{\alpha})\subseteq X_{+}$ is $u\tau$-null. Note that  we can always assume that the net $(y_{\alpha})$ is positive. Therefore, we have
\[S(y_{\alpha})=\sup_{\beta}T(y_{\alpha}\wedge b_{\beta})\leq T(y_{\alpha})\rightarrow 0,\]
in which $(b_{\beta})$ is a positive net in $X$ which is convergent to zero in the $u\tau$-topology.
\end{proof}
\begin{rem}
One can see directly that if a positive operator $T$ is dominated by a positive $u\tau$-Dunford-Pettis operator $S$, then $T$ is also $u\tau$-Dunford-Pettis. This does not hold for unbounded compact operators as shown by \cite[Example 9.7]{KMT}. So, from former statement, we conclude that if $T$ and $S$ are two positive $u\tau$-Dunford-Pettis operators, so is $T\vee S$.
\end{rem}

\end{document}